\documentclass[a4paper,10pt,english]{smfart}
\usepackage[english]{babel}
\usepackage[OT2,T1]{fontenc}
\DeclareSymbolFont{cyrletters}{OT2}{wncyr}{m}{n}
\DeclareMathSymbol{\Sha}{\mathalpha}{cyrletters}{"58}
\usepackage{lmodern}
\usepackage{graphicx}
\usepackage{tabularx}
\usepackage{amsmath,amsthm,amssymb,amsfonts}
\usepackage[a4paper,left=4cm,right=4cm,top=4cm,bottom=4cm]{geometry}
\numberwithin{equation}{section}
\title[Average rank in families of quadratic twists]{Average rank in families of quadratic twists: a geometric point of view}
\author{Pierre Le Boudec}
\subjclass{$11$D$45$, $11$G$05$, $14$G$05$}
\address{EPFL SB MATHGEOM TAN \\ MA C$3$ $604$ (B\^{a}timent MA) \\ Station $8$ \\ \text{CH-$1015$} Lausanne \\ Switzerland}
\email{pierre.leboudec@epfl.ch}

\begin{document}

\makeatletter
\def\imod#1{\allowbreak\mkern10mu({\operator@font mod}\,\,#1)}
\makeatother

\newtheorem{lemma}{Lemma}
\newtheorem{theorem}{Theorem}
\newtheorem{corollary}{Corollary}
\newtheorem{proposition}{Proposition}
\newtheorem{conjecture}{Conjecture}
\newtheorem{conj}{Conjecture}
\renewcommand{\theconj}{\Alph{conj}}

\newcommand{\vol}{\operatorname{vol}}
\newcommand{\D}{\mathrm{d}}
\newcommand{\rank}{\operatorname{rank}}
\newcommand{\Pic}{\operatorname{Pic}}
\newcommand{\Gal}{\operatorname{Gal}}
\newcommand{\meas}{\operatorname{meas}}
\newcommand{\Spec}{\operatorname{Spec}}
\newcommand{\eff}{\operatorname{eff}}
\newcommand{\rad}{\operatorname{rad}}
\newcommand{\sq}{\operatorname{sq}}
\newcommand{\tors}{\operatorname{tors}}
\newcommand{\Cl}{\operatorname{Cl}}
\newcommand{\Reg}{\operatorname{Reg}}

\begin{abstract}
We investigate the average rank in the family of quadratic twists of a given elliptic curve defined over $\mathbb{Q}$, when the curves are ordered using the canonical height of their lowest non-torsion rational point.
\end{abstract}

\maketitle

\tableofcontents

\section{Introduction}

Let $E$ be the elliptic curve defined over $\mathbb{Q}$ by the Weierstrass equation
\begin{equation*}
y^2 = x^3 + A x + B,
\end{equation*}
where $(A,B) \in \mathbb{Z}^2$ satisfies $4 A^3 + 27 B^2 \neq 0$.

We extend the set of definition of the M\"{o}bius function to $\mathbb{Z}$ by setting $\mu(d) = \mu(-d)$ for $d < 0$ and $\mu(0) = 0$.  For every integer $d \in \mathbb{Z}$ with $|\mu(d)| = 1$, we denote by $E_d$ the quadratic twist of $E$ defined over $\mathbb{Q}$ by the equation
\begin{equation*}
d y^2 = x^3 + A x + B.
\end{equation*}
We insist on the fact that all along this article, we view $A$ and $B$ as being fixed, and $d$ as a varying parameter. In particular, the dependences on $A$ and $B$ of the constants involved in the notations $O$, $\ll$ and $\gg$ will not be specified.

Let $L(E_d,s)$ denote the Hasse-Weil $L$-function associated to the curve $E_d$. We let
$\rank_{\textrm{an}} E_d(\mathbb{Q})$ be the analytic rank of $E_d$ over $\mathbb{Q}$, that is the order of the zero of $L(E_d,s)$ at $1$. We also let $\rank E_d(\mathbb{Q})$ be the algebraic rank of $E_d$ over
$\mathbb{Q}$, that is the rank of the finitely generated abelian group $E(\mathbb{Q})$. Recall that the Parity Conjecture states that the two integers $\rank_{\textrm{an}} E_d(\mathbb{Q})$ and $\rank E_d(\mathbb{Q})$ have same parity, and the Birch and Swinnerton-Dyer Conjecture asserts that they are actually equal. 

We define the set
\begin{equation*}
\mathcal{S}(X) = \{ d \in \mathbb{Z}, |\mu(d)| = 1, |d| \leq X \}.
\end{equation*}
The conjecture of Goldfeld (see \cite{MR564926}) on the average size of $\rank E_d(\mathbb{Q})$ as $d$ runs over
$\mathcal{S}(X)$, and the Parity Conjecture imply that, for $\iota \in \{ 0, 1 \}$, we have
\begin{equation}
\label{rank at most 1}
\# \{d \in \mathcal{S}(X), \rank E_d(\mathbb{Q}) = \iota \} \sim \frac1{2} \# \mathcal{S}(X),
\end{equation}
and
\begin{equation}
\label{rank at least 2}
\# \{d \in \mathcal{S}(X), \rank E_d(\mathbb{Q}) \geq 2 \} = o(X).
\end{equation}
The estimates \eqref{rank at most 1} and \eqref{rank at least 2} are supported by the Katz-Sarnak Philosophy (see \cite{MR1659828}) about zeros of $L$-functions and also by Random Matrix Theory heuristics (see for instance \cite{MR1956231}). Unfortunately, they are expected to be currently out of reach. Indeed, we only know that the average rank is at least $1/2$ under the Parity Conjecture. However, we note that Heath-Brown \cite{MR1193603,MR1292115} has proved that in the case $(A,B) = (-1,0)$, the average rank is bounded by an explicit constant.

Ordering the family of curves $E_d$ using the parameter $d$, or any object among the discriminant, the coefficients of the minimal Weierstrass equation, the conductor, essentially amounts to the same thing, so that for each of these orderings, the situation is expected to be similar. The purpose of this article is to investigate this problem from a fundamentally different point of view, that is using a geometric flavoured ordering.

The main geometric invariant of an elliptic curve is its regulator. However, the regulator is defined as being equal to $1$ for rank $0$ curves, so that the number of curves $E_d$ with bounded regulator is infinite. In addition, noticing that the regulator is essentially the product of the canonical heights of the elements of an almost orthogonal basis of the Mordell-Weil lattice (see \cite[Theorem $4$.$2$]{MR717593}), we see that a bound on the regulator implies a much stronger bound on the canonical height of the lowest non-torsion rational point as soon as $\rank E_d(\mathbb{Q}) \geq 2$. Thus, the regulator does not seem to be well-suited for our purpose, and we introduce another geometric invariant.

Let $\hat{h}_{E_d}$ be the canonical height on the curve $E_d$ (see Section \ref{Section Preliminaries} for its definition), and let $E_d(\mathbb{Q})_{\tors}$ be the torsion part of the abelian group $ E_d(\mathbb{Q})$. The author \cite{LeBoudec} has recently investigated the quantity $\eta_d(A,B)$ defined by
\begin{equation*}
\log \eta_d(A,B) = \min \{ \hat{h}_{E_d}(P), P \in E_d(\mathbb{Q}) \smallsetminus E_d(\mathbb{Q})_{\tors} \},
\end{equation*}
if $\rank E_d(\mathbb{Q}) \geq 1$ and $\eta_d(A,B) = \infty$ if $\rank E_d(\mathbb{Q}) = 0$. Let us define the set
\begin{equation*}
\mathcal{H}(Y) = \{d \in \mathbb{Z}, |\mu(d)| = 1, \eta_d(A,B) \leq Y \}.
\end{equation*}
We will prove that $\mathcal{H}(Y)$ is a finite set. We can thus define
\begin{equation*}
\mathcal{AR}_{\textrm{an}}(Y) =
\frac1{\# \mathcal{H}(Y)} \sum_{d \in \mathcal{H}(Y)} \rank_{\textrm{an}} E_d(\mathbb{Q}),
\end{equation*}
and similarly
\begin{equation*}
\mathcal{AR}(Y)= \frac1{\# \mathcal{H}(Y)} \sum_{d \in \mathcal{H}(Y)} \rank E_d(\mathbb{Q}).
\end{equation*}

Note that if $d \in \mathcal{H}(Y)$, then we clearly have $\rank E_d(\mathbb{Q}) \geq 1$. In addition, we also have
$\rank_{\textrm{an}} E_d(\mathbb{Q}) \geq 1$. Indeed, if we had $\rank_{\textrm{an}} E_d(\mathbb{Q}) = 0$, then by the works of Kolyvagin \cite{MR954295} and Breuil, Conrad, Diamond and Taylor \cite{MR1839918}, we would also have
$\rank E_d(\mathbb{Q}) = 0$, which would contradict the fact that $d \in \mathcal{H}(Y)$. Therefore, using our ordering, we have discarded the curves with (analytic or algebraic) rank $0$. As a result, the expectations \eqref{rank at most 1} and \eqref{rank at least 2} might lead to the belief that the average rank should be expected to be equal to $1$. We prove that this is not the case.

\begin{theorem}
\label{Theorem Analytic}
We have the lower bound
\begin{equation*}
\liminf_{Y \to \infty} \mathcal{AR}_{\textrm{an}}(Y) > 1.
\end{equation*}
\end{theorem}

We actually prove that, as $d$ runs over $\mathcal{H}(Y)$, there is a positive proportion of curves $E_d$ which have analytic rank even and at least $2$ (and also a positive proportion of curves $E_d$ which have odd analytic rank). Therefore, we also get the following result.

\begin{theorem}
\label{Theorem Algebraic}
Assume the Parity Conjecture. We have the lower bound
\begin{equation*}
\liminf_{Y \to \infty} \mathcal{AR}(Y) > 1.
\end{equation*}
\end{theorem}

The subfamily of curves with analytic rank even and at least $2$ that we use, was first considered by Gouv\^{e}a and Mazur \cite{MR1080648} to prove that
\begin{equation*}
\# \{d \in \mathcal{S}(X), \rank_{\textrm{an}} E_d(\mathbb{Q}) \geq 2 \} \gg X^{1/2 - \varepsilon},
\end{equation*}
for any fixed $\varepsilon > 0$. We note that using Lemma \ref{Lemma R}, one can show that this estimate actually holds with
$\varepsilon = 0$.

The two main tools used to prove Theorems \ref{Theorem Analytic} and \ref{Theorem Algebraic} are the squarefree sieve of Gouv\^{e}a and Mazur \cite{MR1080648} for binary quartic forms (studied independently by Greaves \cite{MR1150469}), and an upper bound of Heath-Brown \cite[Theorem $10$]{MR1906595} for the number of rational points of bounded height on smooth surfaces.

It is interesting to ask what can be established unconditionally about the quantity
$\# \{d \in \mathcal{H}(Y), \rank E_d(\mathbb{Q}) \geq 2 \}$. Following the lines of the proofs of
Theorems \ref{Theorem Analytic} and~\ref{Theorem Algebraic}, and using the work of Stewart and Top
\cite[Proof of Theorem $3$]{MR1290234} instead of \cite{MR1080648}, one can check that
\begin{equation*}
\# \{d \in \mathcal{H}(Y), \rank E_d(\mathbb{Q}) \geq 2 \} \gg \# \mathcal{H}(Y)^{1/4}.
\end{equation*}

We finish this introduction by speculating on the distribution of $\rank E_d(\mathbb{Q})$ as $d$ runs over
$\mathcal{H}(Y)$. Recall that the sign $\omega(E_d)$ of the functional equation of $L(E_d,s)$ is determined by congruence conditions modulo the conductor of the curve $E$. Therefore, it seems reasonable to expect that $\omega(E_d)$ should be equidistributed as $d$ runs over $\mathcal{H}(Y)$. Recall that if $d \in \mathcal{H}(Y)$ then
$\rank_{\textrm{an}} E_d(\mathbb{Q}) \geq 1$. This implies that if we also have $\omega(E_d) = 1$ then
$\rank_{\textrm{an}} E_d(\mathbb{Q})$ is even and at least $2$. Following the general expectation that the analytic rank should in general be as small as possible, and the Birch and Swinnerton-Dyer Conjecture, we propose the following conjecture.

\begin{conjecture}
\label{Conjecture}
For $\iota \in \{1, 2 \}$, we have the estimate
\begin{equation*}
\# \{d \in \mathcal{H}(Y), \rank E_d(\mathbb{Q}) = \iota \} \sim \frac1{2} \# \mathcal{H}(Y).
\end{equation*}
\end{conjecture}

Conjecture \ref{Conjecture} immediately implies that
\begin{equation*}
\# \{d \in \mathcal{H}(Y), \rank E_d(\mathbb{Q}) \geq 3 \} = o ( \# \mathcal{H}(Y) ),
\end{equation*}
and
\begin{equation*}
\lim_{Y \to \infty} \mathcal{AR}(Y) = \frac{3}{2}.
\end{equation*}

We note that, in Theorems \ref{Theorem Analytic} and \ref{Theorem Algebraic}, we have not tried to obtain lower bounds with explicit dependences on $A$ and $B$. However, that could easily be done since everything amounts to using explicit inequalities for the difference between the Weil height and the canonical height (see for instance \cite{MR1035944}).

\subsection*{Acknowledgements}

It is a great pleasure for the author to thank Peter Sarnak for stimulating conversations related to the topics of this article.

The financial support and the perfect working conditions provided by the \'{E}cole Polytechnique F\'{e}d\'{e}rale de Lausanne are gratefully acknowledged.

\section{Preliminaries}

\label{Section Preliminaries}

This section is devoted to the proofs of three preliminary results which will be the key tools in the proofs of
Theorems \ref{Theorem Analytic} and \ref{Theorem Algebraic}. 

We start by recalling the definition of the Weil height $h_x : \mathbb{P}^2(\mathbb{Q}) \to \mathbb{R}_{\geq 0}$. For
$P \in \mathbb{P}^2(\mathbb{Q})$ with coordinates $(x:y:z)$ where $x, y, z \in \mathbb{Z}$ are such that
$\gcd(x,y,z) = 1$, we set
\begin{equation*}
h_x(P) = \log \max \{ |x|, |z| \}
\end{equation*}
if $(x:y:z) \neq (0:1:0)$ and $h_x(0:1:0) = 0$. In addition, we also recall the definition of the canonical height
$\hat{h}_{E_d} : E_d(\mathbb{Q}) \to \mathbb{R}_{\geq 0}$ on the curve $E_d$. For $P \in E_d(\mathbb{Q})$, we set
\begin{equation*}
\hat{h}_{E_d}(P) = \frac1{2} \lim_{n \to \infty} 4^{-n} h_x(2^n P),
\end{equation*}
where $2^n P$ denotes the point obtained by adding the point $P$ to itself $2^n$ times using the group law of $E_d$.

We can compare the two heights $h_x$ and $\hat{h}_{E_d}$ as follows. The recent work of the author
\cite[Lemma $3$]{LeBoudec} states that, for $P \in E_d(\mathbb{Q})$, we have
\begin{equation}
\label{heights}
\hat{h}_{E_d}(P) = \frac1{2} h_x(P) + O(1),
\end{equation}
where the constant involved in the notation $O$ may depend on $E$ but neither on the point $P$ nor on the integer $d$.

First, we give a sharp upper bound for the cardinality of $\mathcal{H}(Y)$.

\begin{lemma}
\label{Lemma H}
We have the upper bound
\begin{equation*}
\# \mathcal{H}(Y) \ll Y^4.
\end{equation*}
\end{lemma}

\begin{proof}
We have
\begin{equation*}
\# \mathcal{H}(Y) \leq \sum_{d \in \mathbb{Z}} |\mu(d)|
\# \{ P \in E_d(\mathbb{Q}) \smallsetminus E_d(\mathbb{Q})_{\tors}, \exp \hat{h}_{E_d}(P) \leq Y \}.
\end{equation*}
By the estimate \eqref{heights}, we also have
\begin{equation*}
\# \mathcal{H}(Y) \leq \sum_{d \in \mathbb{Z}} |\mu(d)|
\# \{ P \in E_d(\mathbb{Q}) \smallsetminus E_d(\mathbb{Q})_{\tors}, \exp h_x(P) \ll Y^2 \}.
\end{equation*}
We note that if $P \in E_d(\mathbb{Q}) \smallsetminus E_d(\mathbb{Q})_{\tors}$ has coordinates
$(x:y:z) \in \mathbb{P}^2(\mathbb{Q})$, then $z \neq 0$ since $P$ is not the point at infinity, and also $y \neq 0$ since $P$ is not a $2$-torsion point. In addition, we recall that exchanging $P$ and $-P$ amounts to multiplying $y$ by $-1$. Finally, we note that the statement of \cite[Lemma $1$]{LeBoudec} also holds if we respectively replace the assumption $d \geq 1$ by $d \neq 0$, and the conclusion $d_0 \geq 1$ by $d_0 \neq 0$. Therefore, using
\cite[Lemma $1$]{LeBoudec}, we deduce
\begin{equation*}
\# \mathcal{H}(Y) \leq
2 \# \left\{ (d_0, d_1, b_1, y, x_1) \in \mathbb{Z} \times \mathbb{Z}_{\geq 1}^3 \times \mathbb{Z},
\begin{array}{l}
|\mu(d_0 d_1)| = 1 \\
\gcd(x_1, d_1 b_1) = 1 \\
d_0 y^2 = x_1^3 + A x_1 d_1^2 b_1^4 + B d_1^3 b_1^6 \\
|x_1|, d_1 b_1^2 \ll Y^2
\end{array}
\right\}.
\end{equation*}
This gives
\begin{equation*}
\# \mathcal{H}(Y) \leq 2 \sum_{|x_1|, d_1 b_1^2 \ll Y^2}
\# \left\{ (d_0, y) \in \mathbb{Z} \times \mathbb{Z}_{\geq 1},
\begin{array}{l}
|\mu(d_0)| = 1 \\
d_0 y^2 = x_1^3 + A x_1 d_1^2 b_1^4 + B d_1^3 b_1^6
\end{array}
\right\}.
\end{equation*}
For fixed $(d_1, b_1, x_1) \in \mathbb{Z}_{\geq 1}^2 \times \mathbb{Z}$, the cardinality in the right-hand side is at most $1$, so we get
\begin{equation*}
\# \mathcal{H}(Y) \ll Y^4,
\end{equation*}
as wished.
\end{proof}

Recall that $(A,B) \in \mathbb{Z}^2$ satisfies $4 A^3 + 27 B^2 \neq 0$. We introduce the binary quartic form
\begin{equation*}
Q(u,v) = u (v^3 + A u^2 v + B u^3),
\end{equation*}
and also, for $d \in \mathbb{Z}$, we set
\begin{equation}
\label{definition r_Q}
r_Q(d; Z) = \# \left\{ |u|, |v| \leq Z, Q(u,v) = d \right\}.
\end{equation}
We need to investigate the second moment of $r_Q(d; Z)$. We thus introduce
\begin{equation}
\label{definition R_Q}
\mathcal{R}_Q(Z) = \sum_{d \in \mathbb{Z}} r_Q(d; Z)^2.
\end{equation}

We now establish a sharp upper bound for the quantity $\mathcal{R}_Q(Z)$.

\begin{lemma}
\label{Lemma R}
We have the upper bound
\begin{equation*}
\mathcal{R}_Q(Z) \ll Z^2.
\end{equation*}
\end{lemma}

\begin{proof}
We have
\begin{equation*}
\mathcal{R}_Q(Z) = \# \left\{ |u|, |v|, |w|, |t| \leq Z, Q(u,v) = Q(w,t) \right\},
\end{equation*}
and thus
\begin{equation*}
\mathcal{R}_Q(Z) = 1 + \sum_{k \leq Z}
\# \left\{ |u_0|, |v_0|, |w_0|, |t_0| \leq Z_0,
\begin{array}{l}
Q(u_0, v_0) = Q(w_0, t_0) \\
\gcd(u_0, v_0, w_0, t_0) = 1
\end{array}
\right\},
\end{equation*}
where we have set $Z_0 = Z/k$. Let us now introduce the exponential naive height
$H : \mathbb{P}^3(\mathbb{Q}) \to \mathbb{R}_{> 0}$. For $\mathbf{x} = (x_0 : x_1 : x_2 : x_3) \in \mathbb{P}^3(\mathbb{Q})$, where $x_0, x_1, x_2, x_3 \in \mathbb{Z}$ are such that $\gcd(x_0, x_1, x_2, x_3) = 1$, we set
\begin{equation*}
H(\mathbf{x}) = \max \{ |x_0|, |x_1|, |x_2|, |x_3| \}.
\end{equation*}
Let us also denote by $V$ the quartic surface defined by the equation 
\begin{equation*}
Q(x_0,x_1) - Q(x_2,x_3)= 0.
\end{equation*}
Using this notation, we have
\begin{equation*}
\mathcal{R}_Q(Z) = 1 + 2 \sum_{k \leq Z} \# \left\{ \mathbf{x} \in V(\mathbb{Q}),  H(\mathbf{x}) \leq Z_0 \right\}.
\end{equation*}
Let $U$ be the Zariski open subset of $V$ defined by removing the lines from $V$. Since $4 A^3 + 27 B^2 \neq 0$, it is easy to check that $V$ is a smooth surface. Thus, the result of Heath-Brown \cite[Theorem $10$]{MR1906595} gives
\begin{equation*}
\# \left\{ \mathbf{x} \in U(\mathbb{Q}),  H(\mathbf{x}) \leq Z_0 \right\} \ll Z_0^{2 - 2/9 + \varepsilon},
\end{equation*}
for any fixed $\varepsilon > 0$. In addition, the contribution of the rational points on a line of $V$ is $\ll Z_0^2$. Since the number of lines contained in $V$ is finite, we have
\begin{equation*}
\# \left\{ \mathbf{x} \in (V \smallsetminus U)(\mathbb{Q}),  H(\mathbf{x}) \leq Z_0 \right\} \ll Z_0^2.
\end{equation*}
We finally obtain
\begin{equation*}
\mathcal{R}_Q(Z) \ll \sum_{k \leq Z} \left( \frac{Z}{k} \right)^2,
\end{equation*}
which completes the proof.
\end{proof}

It is worth noting that the constant involved in the upper bound of Heath-Brown \cite[Theorem $10$]{MR1906595} is independent of $Q$. In addition, a result of Colliot-Th\'{e}l\`{e}ne \cite[Appendix, Theorem]{MR1906595} implies that the number of lines contained in $V$ is bounded by a constant independent of $Q$. As a result, the constant involved in the upper bound in Lemma \ref{Lemma R} is also independent of $Q$.

Recall that $\omega(E_d)$ denotes the sign of the functional equation of $L(E_d,s)$. For $\nu \in \{ -1, 1 \}$, we introduce
\begin{equation}
\label{definition S_Q}
\mathcal{S}_{Q,\nu}(Z) = \sum_{\substack{d \in \mathbb{Z} \\ \omega(E_d) = \nu}} |\mu(d)| r_Q(d; Z).
\end{equation}

We state a sharp lower bound for the quantity $\mathcal{S}_{Q,\nu}(Z)$ under a mild assumption which is not restrictive in practice.

\begin{lemma}
\label{Lemma S}
Let $\mathcal{N}_E$ be the conductor of the curve $E$. Assume that $A, B \mid 12 \mathcal{N}_E$. We have the lower bound
\begin{equation*}
\mathcal{S}_{Q,\nu}(Z) \gg Z^2.
\end{equation*}
\end{lemma}

\begin{proof}
We have
\begin{equation*}
\mathcal{S}_{Q,\nu}(Z) = \sum_{\substack{|u|, |v| \leq Z \\ \omega(E_{Q(u,v)}) = \nu}} |\mu(Q(u,v))|.
\end{equation*}
In \cite[Proposition $6$]{MR1080648}, Gouv\^{e}a and Mazur use their squarefree sieve for binary quartic forms to prove the existence of an explicit constant $c > 0$ such that
\begin{equation*}
\mathcal{S}_{Q,\nu}(Z) \geq c Z^2 + O \left( \frac{Z}{(\log Z)^{1/2}} \right),
\end{equation*}
which completes the proof.
\end{proof}

\section{Proofs of Theorems \ref{Theorem Analytic} and \ref{Theorem Algebraic}}

Recall that $\mathcal{N}_E$ denotes the conductor of the curve $E$ and let $M = 12 \mathcal{N}_E$. We can assume without loss of generality that $E$ is given by a model
\begin{equation*}
y^2 = x^3 + A x + B,
\end{equation*}
for which we have $A, B \mid M$. Indeed, this can be seen by making the substitution $(x,y) \mapsto (x/M^2, y/M^3)$, and then clearing denominators.

We prove that, as $d$ runs over $\mathcal{H}(Y)$, there is a positive proportion of curves $E_d$ for which
$\rank_{\textrm{an}} E_d(\mathbb{Q})$ is even and at least $2$, and also a positive proportion of curves $E_d$ for which $\rank_{\textrm{an}} E_d(\mathbb{Q})$ is odd. This clearly implies Theorems \ref{Theorem Analytic} and \ref{Theorem Algebraic}.

As already mentioned, if $d \in \mathcal{H}(Y)$ then $\rank_{\textrm{an}} E_d(\mathbb{Q})$ is even and at least $2$ if and only if $\omega(E_d) = 1$. For $\nu \in \{ -1, 1 \}$, we set
\begin{equation*}
\Omega_{\nu}(Y) = \frac1{\# \mathcal{H}(Y)}
\sum_{\substack{d \in \mathcal{H}(Y) \\ \omega(E_d) = \nu}} 1.
\end{equation*}
Our goal is thus to prove that 
\begin{equation}
\label{result Omega}
\liminf_{Y \to \infty} \Omega_{\nu}(Y) > 0.
\end{equation}

First, we use Lemma \ref{Lemma H} to obtain
\begin{equation}
\label{lower bound 1}
\Omega_{\nu}(Y) \gg \frac1{Y^4}
\sum_{\substack{d \in \mathcal{H}(Y) \\ \omega(E_d) = \nu}} 1.
\end{equation}

We note that if $(u,v) \in \mathbb{Z}^2$ is such that $|\mu(Q(u,v))| = 1$ then the point $P$ with coordinates $(uv:1:u^2)$ satisfies $P \in E_{Q(u,v)}(\mathbb{Q})$. By \cite[Lemma $3$]{LeBoudec}, we have
\begin{equation}
\label{heights 2}
\hat{h}_{E_{Q(u,v)}}(P) = \frac1{2} h_x(P) + O(1),
\end{equation}
where the constant involved in the notation $O$ may depend on $E$ but not on $(u,v)$. The assumption
$|\mu(Q(u,v))| = 1$ implies that $u$ and $v$ are coprime and thus
\begin{equation*}
h_x(P) = \log \max \{ |u|, |v| \}.
\end{equation*}
As a result, if $\max \{ |u|, |v| \}$ is larger than a constant which depends at most on $E$, then the estimate \eqref{heights 2} implies that $\hat{h}_{E_{Q(u,v)}}(P) \neq 0$ and thus $P$ is not a torsion point of
$E_{Q(u,v)}(\mathbb{Q})$. Therefore, we have
\begin{equation*}
\log \eta_{Q(u,v)}(A,B) \leq \hat{h}_{E_{Q(u,v)}}(P),
\end{equation*}
which gives
\begin{equation*}
\eta_{Q(u,v)}(A,B) \ll \max \{ |u|^{1/2}, |v|^{1/2} \},
\end{equation*}
where the constant involved in the notation $\ll$ depends at most on $E$. Therefore, we have proved that there exists a constant $C_E > 0$ such that if $(u,v) \in \mathbb{Z}^2$ satisfies $|\mu(Q(u,v))| = 1$ and $|u|, |v| \leq C_E Y^2$, then
$\eta_{Q(u,v)}(A,B) \leq Y$, apart from a certain number of exceptions bounded only in terms of $E$. In other words, if we introduce the set
\begin{equation*}
\mathcal{Q}(Y) = \left\{d \in \mathbb{Z}, |\mu(d)| = 1, r_Q(d; C_E Y^2) \geq 1 \right\},
\end{equation*}
where $r_Q(d; C_E Y^2)$ is defined in \eqref{definition r_Q}, we have
\begin{equation*}
\# \left( \mathcal{Q}(Y) \smallsetminus \mathcal{H}(Y) \right) \ll 1,
\end{equation*}
where the constant involved in the notation $\ll$ depends at most on $E$.

Recalling the lower bound \eqref{lower bound 1}, we get
\begin{equation}
\label{lower bound 2}
\Omega_{\nu}(Y) \gg \frac1{Y^4}
\sum_{\substack{d \in \mathcal{Q}(Y) \\ \omega(E_d) = \nu}} 1 + O \left( \frac1{Y^4} \right).
\end{equation}
Recall the respective definitions \eqref{definition R_Q} and \eqref{definition S_Q} of $\mathcal{R}_Q(Z)$ and
$\mathcal{S}_{Q, \nu}(Z)$. We use the Cauchy-Schwarz inequality to obtain
\begin{equation}
\label{lower bound 3}
\sum_{\substack{d \in \mathcal{Q}(Y) \\ \omega(E_d) = \nu}} 1 \geq
\frac{\mathcal{S}_{Q, \nu}(C_E Y^2)^2}{\mathcal{R}_Q(C_E Y^2)}.
\end{equation}
Putting together the lower bounds \eqref{lower bound 2} and \eqref{lower bound 3}, we get
\begin{equation*}
\Omega_{\nu}(Y) \gg \frac{\mathcal{S}_{Q, \nu}(C_E Y^2)^2}{Y^4 \mathcal{R}_Q(C_E Y^2)}
+ O \left( \frac1{Y^4} \right).
\end{equation*}
Lemma \ref{Lemma R} states that $\mathcal{R}_Q(C_E Y^2) \ll Y^4$. In addition, since we have assumed that
$A, B \mid 12 \mathcal{N}_E$, we can use Lemma \ref{Lemma S} to obtain $\mathcal{S}_{Q, \nu}(C_E Y^2) \gg Y^4$. Therefore, we deduce the existence of a constant $c_E > 0$, depending at most on $E$, such that
\begin{equation*}
\Omega_{\nu}(Y) \geq c_E,
\end{equation*}
which completes the proof of the lower bound \eqref{result Omega}, and thus also the proofs of
Theorems \ref{Theorem Analytic} and \ref{Theorem Algebraic}.

\bibliographystyle{amsalpha}
\bibliography{biblio}

\end{document}